\documentclass[a4paper,11pt]{amsart}
\usepackage{latexsym, amssymb,amsmath,amsthm}
\usepackage[all, knot]{xy}
\xyoption{arc}

\def \To{\longrightarrow}
\def \dim{\operatorname{dim}}

\def \C{\mathcal{C}}
\def \D{\Delta}
\def \d{\delta}
\def \dl{\d_{_L}}

\def \e{\varepsilon}

\def \S{\mathcal{S}}

\def \Z{\mathbb{Z}}

\numberwithin{equation}{section}

\newtheorem{theorem}{Theorem}[section]
\newtheorem{lemma}[theorem]{Lemma}
\newtheorem{proposition}[theorem]{Proposition}
\newtheorem{corollary}[theorem]{Corollary}
\newtheorem{definition}[theorem]{Definition}
\newtheorem{example}[theorem]{Example}
\newtheorem{remark}[theorem]{Remark}
\newtheorem{convention}[theorem]{Convention}

\begin{document}

\title[FROM PROJECTIVE REPRESENTATIONS TO QUASI-QUANTUM GROUPS]
{FROM PROJECTIVE REPRESENTATIONS \\ TO QUASI-QUANTUM GROUPS}
\author[H.-L. HUANG]{Hua-Lin Huang}
\address{School of Mathematics, Shandong University, Jinan 250100, P. R. China}
\email{hualin@sdu.edu.cn}
\date{}
\maketitle

\begin{abstract}
This is a contribution to the project of quiver approaches to
quasi-quantum groups initiated in \cite{qha1}. We classify Majid
bimodules over groups with 3-cocycles by virtue of projective
representations. This leads to a theoretic classification of graded
pointed Majid algebras over path coalgebras, or equivalently cofree
pointed coalgebras, and helps to provide a projective
representation-theoretic description of the gauge equivalence of
graded pointed Majid algebras. We apply this machinery to construct
some concrete examples and obtain a classification of
finite-dimensional graded pointed Majid algebras with the set of
group-likes equal to the cyclic group $\Z_2$ of order 2.

\vskip 5pt

\noindent{\bf 2000 MR Subject Classification} \ \ 16W30, 16W35,
20C25

\end{abstract}

\section{Introduction}

Quasi-quantum groups, that is quasi-Hopf algebras, were introduced
by Drinfeld \cite{d} in accordance with his philosophy of quantum
groups. Since its appearance, the notion has been playing important
roles in various areas of mathematics and physics. Our goal is to
classify some interesting classes of quasi-quantum groups and their
representations. Since the classification problem of general
quasi-quantum groups is still very complicated, we focus on
elementary (i.e., finite-dimensional and the simple modules of the
underlying algebras are 1-dimensional) quasi-Hopf algebras and
pointed (i.e., the simple comodules of the underlying coalgebras are
1-dimensional) Majid algebras (=dual quasi-Hopf algebras). In this
situation we can take advantage of the well-developed representation
theory of algebras (see for instance \cite{ass}), especially the
quiver techniques.

A quiver setting for quasi-quantum groups is provided in
\cite{qha1}. It is shown that a systematic study of elementary
quasi-Hopf algebras or pointed Majid algebras, in particular the
classification problem and the associated representation theory, can
be carried out effectively in this framework by virtue of the quiver
techniques. Under our quiver setting the study of elementary
quasi-Hopf algebras can be included in that of pointed Majid
algebras, so we always work on the latter only. The present paper is
devoted to this quiver classification project of pointed Majid
algebras.

For an arbitrary pointed Majid algebra, one can always associate to
it a coradically graded version induced by its coradical filtration.
So in this paper, by graded we actually mean coradically graded.
According to a Gabriel type theorem for pointed Majid algebras
\cite{qha1}, any graded pointed Majid algebra can be embedded into a
graded Majid algebra structure on some unique Hopf quiver. Therefore
the first step of the classification project is to investigate the
graded Majid algebra structures on Hopf quivers. The aim of this
paper is to give an explicit classification of such Majid algebras.

Let $k$ be a field. For a Hopf quiver $Q,$ a graded Majid algebra
structure on the path coalgebra $kQ$ determines on its set of
vertices $Q_0$ a group structure $G$ with a 3-cocycle $\Phi$ and a
$(kG,\Phi)$-Majid bimodule structure on the space $kQ_1$ spaned by
the set of arrows. Conversely, given a group $G$ with a 3-cocycle
$\Phi$ and a $(kG,\Phi)$-Majid bimodule $M,$ we can associate to the
data a Hopf quiver $Q$ and a graded Majid algebra structure on the
path coalgebra $kQ.$ See \cite{qha1} for the precise description. So
for our purpose, it suffices to classify all $(kG,\Phi)$-Majid
bimodules for a general group $G$ and an arbitrary 3-cocycle $\Phi.$

The notion of Majid bimodules is certainly an analog of the familiar
one of Hopf bimodules. Nichols \cite{nichols} introduced the
definition of Hopf bimodules and initiated the classification over
finite abelian groups. In \cite{cr1, cr2}, Cibils and Rosso gave an
explicit classification of the Hopf bimodules over an arbitrary
group and applied the result to define Hopf quivers and provide the
complete classification of graded Hopf structures on path
(co)algebras. In more detail, for a group $G,$ the category of
$kG$-Hopf bimodules is equivalent to the product of the categories
of usual module categories $\prod_{C \in \C} kZ_C\mathrm{-mod},$
where $\C$ is the set of conjugacy classes and $Z_C$ is the
centralizer of one of the elements in the class $C \in \C;$ and via
the quantum shuffle product \cite{rosso1} on Hopf quivers determined
by these data of $kG$-Hopf bimodules, a complete list of graded
cofree pointed Hopf algebras is obtained. Here by ``cofree" we mean
the underlying coalgebras are cotensor coalgebras and therefore they
enjoy the universal mapping property (see \cite{nichols}).

Our main result is an extension of Cibils and Rosso's works to Majid
bimodules and Majid algebras. It turns out that the category of
$(kG,\Phi)$-Majid bimodules is equivalent to the product of
categories $\prod_{C \in \C} (kZ_C,\Tilde{\Phi}_C)\mathrm{-rep},$
where $\C, Z_C$ are as before and $\Tilde{\Phi}_C$ is a 2-cocycle on
$Z_C$ determined by $\Phi$ and $(kZ_C,\Tilde{\Phi}_C)\mathrm{-rep}$
is the category of projective $\Tilde{\Phi}_C$-representations in
the sense of Schur, or equivalently the left module category of the
twisted group algebra $k^{\Tilde{\Phi}_C}Z_C$ (see for
example\cite{karp}). As applications of this result, we obtain a
complete classification of graded Majid algebra structures on the
path coalgebras of Hopf quivers, which amounts to a classification
of graded cofree pointed Majid algebras, and a projective
representation-theoretic interpretation of their gauge equivalence.
We apply this machinery to construct some concrete examples of Majid
bimodules and Majid algebras, in which the group $G$ is the cyclic
group $\Z_2$ of order 2. A classification of finite-dimensional
graded pointed Majid algebras over $\Z_2$ is obtained by making use
of the Gabriel type theorem.

Recently, Etingof and Gelaki have a series of works
\cite{eg1,eg2,eg3,g} on the classification of elementary quasi-Hopf
algebras over the field of complex numbers. By the well-known
Tannakian formalism, their works lead to some interesting
classification results for finite tensor categories (see \cite{ce}).
Their classification method bases on very clever constructions ``by
hand", which is quite different from ours. We will show that our
approach can provide a conceptual understanding for their direct
constructions by the projective representation theory of groups.

The layout of the paper is as follows. In Section 2 we give the
definitions of Majid modules and Majid bimodules and include a
fundamental structure result for Majid modules. In Section 3 we
provide the explicit classification of Majid bimodules over groups
with 3-cocycles. In Section 4 we give, via the classification
result, a theoretic classification of graded Majid algebra
structures over path coalgebras and a description of their gauge
equivalence. In Section 5 we provide some concrete examples and
obtain a classification of finite-dimensional graded pointed Majid
algebras over $\Z_2.$

Throughout, we fix a ground field $k$ and vector spaces, linear
mappings, (co)algebras and unadorned tensor product $\otimes$ are
over $k.$ Unexplained notions about Hopf algebras, quasi-Hopf
algebras and Majid algebras can be found in the books \cite{kassel,
majid, ss}. We turn to \cite{qha1} frequently for notations and
results of the quiver setting of Majid algebras.

\section{Majid Modules and Fundamental Structure}

We recall the notion of Majid bimodules for the convenience of the
reader. For our purpose we also include the notion of Majid modules
and an analog of Larson and Sweedler's Fundamental Theorem for Hopf
modules \cite{ls}.

\begin{definition}
Assume that $H$ is a Majid algebra with reassociator $\Phi.$ A
linear space $M$ is called an $H$-Majid bimodule, if $M$ is an
$H$-bicomodule with structure maps $(\d_{_L},\d_{_R}),$ and there
are two $H$-bicomodule morphisms
$$\rho_{_L}: H \otimes M \To M, \ h \otimes m \mapsto h.m, \quad
\rho_{_R}: M \otimes H \To M, \ m \otimes h \mapsto m.h$$ such that
for all $g, h \in H, m \in M,$ the following equalities hold:
\begin{gather}
1_H.m=m=m.1_H,\\
g_1.(h_1.m_0)\Phi(g_2,h_2,m_1)=\Phi(g_1,h_1,m^{-1})(g_2h_2).m^0, \\
m_0.(g_1h_1)\Phi(m_1,g_2,h_2)=\Phi(m^{-1},g_1,h_1)(m^0.g_2).h_2, \\
g_1.(m_0.h_1)\Phi(g_2,m_1,h_2)=\Phi(g_1,m^{-1},h_1)(g_2.m^0).h_2,
\end{gather}
where we use the Sweedler notation $$\D(g)=g_1 \otimes g_2, \quad
\d_{_L}(m)=m^{-1} \otimes m^0, \quad \d_{_R}(m)=m_0 \otimes m_1$$
for coproduct and comodule structure maps.

Suppose only the $H$-bicomodule morphism $\rho_{_L}$ is defined for
the $H$-bicomodule $M$ satisfying $1_H.m=m$ and (2.2), we call it a
left $H$-Majid module. Similarly we define a right $H$-Majid module.
\end{definition}

For the purpose of the present paper, it is enough to focus on the
case of $H$ being a group algebra with 3-cocycle. This also allows
better exposition without much loss of generality.

From now on, we fix a group $G$ with unit element $\epsilon$ and a
3-cocycle $\Phi: G \times G \times G \To k^*.$ Recall that, $\Phi$
being a 3-cocycle means the following equalities hold for all
$e,f,g,h \in G$:
\begin{gather}
\Phi(e,f,g)\Phi(e,fg,h)\Phi(f,g,h)=\Phi(ef,g,h)\Phi(e,f,gh), \\
\Phi(\epsilon,e,f)=\Phi(e,\epsilon,f)=\Phi(e,f,\epsilon)=1.
\end{gather}
In the literature, such a cocycle is said to be normalized. All the
cocycles in this paper are assumed so. We extend $\Phi,$ without
changing the notation, by linearity to a function on $(kG)^{\otimes
3}$ and understand the group algebra $kG$ as a Majid algebra with
reassociator $\Phi.$

Let $M$ be a $(kG, \Phi)$-Majid bimodule. Then in the first place
the underlying bicomodule structure makes it a $G$-bigraded space
$M=\bigoplus_{g,h \in G} \ ^gM^h$ with $(g,h)$-isotypic component
$$^gM^h=\{ m \in M \ | \ \d_{_L}(m)=g \otimes m, \ \d_{_R}(m)=m
\otimes h \} \ .$$ The quasi-bimodule structure maps satisfy the
following equalities:
\begin{gather}
e.(f.m)=\frac{\Phi(e,f,g)}{\Phi(e,f,h)}(ef).m,\\
(m.e).f=\frac{\Phi(h,e,f)}{\Phi(g,e,f)}m.(ef),\\
(e.m).f=\frac{\Phi(e,h,f)}{\Phi(e,g,f)}e.(m.f),
\end{gather}
for all $e,f,g,h \in G$ and $m \in \ ^gM^h.$

For the classification of Majid bimodules, first we should study
Majid modules and develop an analog of the fundamental structure
theorem for Hopf modules. A right $(kG,\Phi)$-Majid module is a
$kG$-bicomodule $M$ with a right quasi-module satisfying (2.8). A
trivial example of right $(kG,\Phi)$-Majid module is $kG$ with
bicomodule defined by the usual diagonal map and right quasi-module
by the right multiplication. It is also easy to see that for any
left $kG$-comodule $(V,d_{_L})$ we can define a right
$(kG,\Phi)$-Majid module on the tensor space $V \otimes kG,$ with
structure maps given by
\begin{gather}
\d_{_L}(v \otimes g)=hg \otimes v \otimes g \ , \quad \d_{_R}(v
\otimes g)=v \otimes g \otimes g \ , \\
(v \otimes g).f=\Phi^{-1}(h,g,f)v \otimes gf \ ,
\end{gather}
for all $f,g,h \in G$ and $v \in \ ^hV=\{v \in V \ | \ d_{_L}(v)=h
\otimes v \}.$ For any right $kG$-comodule $V,$ a left
$(kG,\Phi)$-Majid module on $kG \otimes V$ can be defined
analogously. Majid modules of this form are said to be trivial.

The fundamental structure of Majid modules is described in the
following.

\begin{proposition}
The category $\mathcal{RM}(kG,\Phi)$ of right $(kG,\Phi)$-Majid
modules and the category $kG\mathrm{-comod}$ of left $kG$-comodules
are equivalent. In particular, right $(kG,\Phi)$-Majid modules are
trivial.
\end{proposition}

\begin{proof}
For each $M \in \mathcal{RM}(kG,\Phi),$ let $M^{\epsilon}=\{ m \in M
\ | \ \d_{_R}(m)=m \otimes \epsilon \}$ be the space of right
coinvariants. It is a sub left $kG$-comodule of $M.$ Define a
functor $\Theta: \mathcal{RM}(kG,\Phi) \To kG\mathrm{-comod}$ by $M
\mapsto M^{\epsilon}.$

For each $V \in kG\mathrm{-comod},$ we consider the right
$(kG,\Phi)$-Majid module structure on $V \otimes kG$ as given by
(2.10)-(2.11). Define a functor $\Xi: kG\mathrm{-comod} \To
\mathcal{RM}(kG,\Phi)$ by $V \mapsto V \otimes kG.$

The verification of these functors actually providing the claimed
equivalence of categories is routine. We only mention the
isomorphism maps of $M$ and $M^{\epsilon} \otimes kG$ in which the
3-cocycle $\Phi$ must get involved. Define $\zeta: \ M \To
M^{\epsilon} \otimes kG$ via $m \mapsto
\frac{\Phi(e,f^{-1},f)}{\Phi(f,f^{-1},f)}m.f^{-1} \otimes f$ for all
$m \in \ ^eM^f.$ For the converse, define $\xi:\ M^{\epsilon}
\otimes kG \To M$ via $ v \otimes g \mapsto v.g.$
\end{proof}

\begin{remark}
If we consider also the monoidal structures on the categories
$\mathcal{RM}(kG,\Phi)$ and $kG\mathrm{-comod},$ then the previous
equivalence is actually monoidal. The result can be generalized to
Majid modules over general Majid algebras with proper modification.
The quasi-Hopf version of the fundamental structure theorem of Hopf
modules is developed in \cite{hn}.
\end{remark}

\section{Majid Bimodules over Groups}

Denote the category of $(kG,\Phi)$-Majid bimodules by
$\mathcal{MB}(kG,\Phi).$  For each $M \in \mathcal{MB}(kG,\Phi),$
let $M=\bigoplus_{g,h \in G} \ ^gM^h$ be the decomposition of
isotypic components. According to the axioms of $(kG,\Phi)$-Majid
bimodules, we have for all $f,g,h \in G,$
\begin{equation}
f. \ ^gM^h \ =  \ ^{fg}M^{fh}, \quad  ^gM^h .f \ = \ ^{gf}M^{hf} \ .
\end{equation}
In particular it follows that for all $x,g,c \in G,$
\begin{equation}
^{g^{-1}cgx}M^x = \ ^{g^{-1}cg}M^\epsilon.x = \ ((g^{-1}. \
^cM^\epsilon).g).x \ .
\end{equation}
It is clear that
\begin{equation}
\dim_k{ ^{g^{-1}cgx}M^x}=\dim_k{^cM^\epsilon} \ .
\end{equation}

Now we are ready to classify Majid bimodules over groups. Firstly,
by Proposition 2.2, we have $M \cong M^{\epsilon} \otimes kG$ as
right $(kG,\Phi)$-Majid modules. The space $M^{\epsilon}$ of right
coinvariants is a sub left comodule, but not closed under the left
quasi-module action $\rho_{_L}.$ However, if we consider the
conjugate action, then by (3.2) we have $(g.v).g^{-1} \in
M^{\epsilon}$ for all $v \in M^{\epsilon}$ and $g \in G.$ For
convenience, we use the notation $g \triangleright v =
(g.v).g^{-1}.$ Of course, in general this action does not make a
usual $G$-representation. It turns out to be quite similar to
Schur's projective representation of $G.$ More precisely, for all
$e, f, g \in G, v \in \ ^gM^{\epsilon},$ we have
\begin{equation}
\epsilon \triangleright v =v, \quad  e \triangleright (f
\triangleright v) = \Tilde{\Phi}(e,f,g) (ef) \triangleright v,
\end{equation}
where
$\Tilde{\Phi}(e,f,g)=\frac{\Phi(e,f,g)\Phi(ef,f^{-1},e^{-1})\Phi(e,fg,f^{-1})}{\Phi(efg,f^{-1},e^{-1})\Phi(e,f,f^{-1})}
\ .$ For each $g,$ we choose for $^gM^\epsilon$ a basis $\{
v_g(\lambda)\}_{\lambda \in \Lambda_g},$ then construct and index a
basis $\{ v_g(\lambda)\}_{g \in G, \lambda \in \Lambda_g}$ for
$M^\epsilon$ accordingly. Let $X: G \To GL(M^\epsilon)$ be the
associated matrix representation of the action $\triangleright$
under this basis. Then we have for all $e, f \in G$
\begin{equation}
X(\epsilon)=\operatorname{Id}_{M^\epsilon}, \quad
X(e)X(f)=\chi(e,f)X(ef),
\end{equation}
where $\chi(e,f)$ is a diagonal matrix, with scalar block
$\Tilde{\Phi}(e,f,f^{-1}e^{-1}gef)\operatorname{Id}_{^gM^\epsilon}$
corresponding to the subset $\{ v_g(\lambda)\}_{\lambda \in
\Lambda_g}$ of basis elements of $M^\epsilon.$ Recall that a mapping
$\rho:G \To GL(V)$ is called a projective representation of $G$ if
there exists a mapping $\alpha: G \times G \To k^*$ such that
$\rho(\epsilon)=\operatorname{Id}_V$ and
$\rho(g)\rho(h)=\alpha(g,h)\rho(gh)$ for all $g,h \in G.$ In the
last equality $\alpha(g,h)$ is understood as a scalar matrix. The
representation is also called an $\alpha$-representation when we
need to specify the mapping. It is clear that $\alpha$ is a
2-cocycle on the group $G.$ If the 2-cocycle $\alpha$ is trivial,
that is $\alpha(g,h)=1$ for all $g,h \in G,$ then of course $\rho$
is a usual linear representation. In the just defined mapping $X: G
\To GL(M^\epsilon)$ the $\chi(e,f)$'s are diagonal matrices
consisting of different scalar blocks, thus in general $X$ is not a
projective representation. However, the scalar blocks vary regularly
via the function $\Tilde{\Phi}$ determined by the 3-cocycle $\Phi,$
so we might call $X$ a $\Phi$-twisted projective representation,
which seems to be a natural generalization of the projective
representation.

Here it is worthy to digress for a moment. Note that for
$M^{\epsilon}$ the compatibility of the conjugate action
$\triangleright$ and the left coaction becomes
\begin{equation}
\d_{_L}(f \triangleright v) = fgf^{-1} \otimes f \triangleright v
\end{equation}
for all $f, g \in G, v \in \ ^gM^{\epsilon}.$ The data
$(M^{\epsilon},\triangleright,\dl)$ satisfying (3.6) can be viewed
as the right candidate of crossed module in the sense of
\cite{yetter}, or the more commonly used term Yetter-Drinfeld
module, for Majid algebras. We call the triple
$(M^{\epsilon},\triangleright,\dl)$ a left twisted crossed
$(kG,\Phi)$-module and it is also how $M^{\epsilon}$ inherits the
left $(kG,\Phi)$-Majid module structure of $M.$

We proceed to further investigation of $M^{\epsilon}$ via the
twisted projective representation $X.$ Let $\C$ be the set of the
conjugacy classes of $G.$ For each $C \in \C,$ let
$^CM^{\epsilon}=\oplus_{g \in C} \ ^gM^{\epsilon}.$ Again by (3.2),
the action $\triangleright$ is closed for each $^CM^{\epsilon},$ so
we have the direct sum decomposition of $M$ into sub twisted
projective representations, and even sub twisted crossed modules:
\begin{equation}
M=\bigoplus_{C \in \C} \ ^CM^{\epsilon} \ .
\end{equation}
Moreover, for each $g \in G,$ let $Z_g$ be its centralizer, then we
can restrict the action $\triangleright$ to $Z_g$ on
$^gM^{\epsilon}.$ Let $X_g: Z_g \To GL({^gM^{\epsilon}})$ denote the
corresponding matrix representation. The following observation is a
key step to our classification result.

\begin{lemma}
$X_g$ is a projective representation of $Z_g$ for all $g \in G.$
\end{lemma}

\begin{proof} By (3.5) we have $X_g(\epsilon)=\operatorname{Id}_{\
^gM^\epsilon}$ and $X_g(e)X_g(f)=\Tilde{\Phi}(e,f,g)X_g(ef)$ for all
$e,f \in Z_g.$ Denote by
\begin{equation}
\Tilde{\Phi}_g: Z_g \times Z_g \To k^*, \quad (e,f) \mapsto
\Tilde{\Phi}(e,f,g)
\end{equation}
the mapping induced by $\Tilde{\Phi}.$ Direct calculation indicates
that $\Tilde{\Phi}_g$ is a 2-cocycle on $Z_g,$ thus $X_g$ is a
projective representation of $Z_g.$
\end{proof}

Since $^CM^{\epsilon} = \oplus_{g \in C} \ ^gM^{\epsilon},$ so it is
of interest to view the twisted projective representation
$^CM^{\epsilon}$ of $G$ as the gluing of the set of the usual
projective representations $\{^gM^{\epsilon} \ | \ g \in C \}$ of
the centralizers $Z_g.$

For each $g \in G,$ let $C_g$ be the conjugacy class containing $g.$
Just as the classical group representation theory, we want to extend
the $\Tilde{\Phi}_g$-representation $^gM^\epsilon$ of $Z_g$ to the
twisted projective representation $^{C_g}M^{\epsilon}$ of $G$ by an
appropriate induction procedure. To this end, we consider a twisted
version of tensor product of $kG,$ viewed as a right
$Z_g$-representation, with the $\Tilde{\Phi}_g$-representation
$^gM^\epsilon$ of $Z_g,$ as the quotient space
$$\frac{kG \otimes_k \ ^gM^{\epsilon}}{<e \otimes f
\triangleright m - \Tilde{\Phi}(e,f,g) ef \otimes m \ | \ \forall \
e \in G, f \in Z_g, m \in \ ^gM^\epsilon>} \ .$$ We denote it by $kG
\hat{\otimes}_{kZ_g} \ ^gM^\epsilon.$ For each $e \otimes m \in kG
\otimes_k \ ^gM^\epsilon,$ let $e \hat{\otimes} m$ be its image in
$kG \hat{\otimes}_{kZ_g} \ ^gM^\epsilon.$ Define two linear mappings
as follows:
\begin{eqnarray}
d_{_L}: kG \hat{\otimes}_{kZ_g} \ ^gM^\epsilon &\To& kG \otimes kG
\hat{\otimes}_{kZ_g} \ ^gM^\epsilon, \\ e \hat{\otimes} m
&\mapsto& ege^{-1} \otimes e \hat{\otimes} m, \nonumber \\
m_{_L}: kG \otimes kG \hat{\otimes}_{kZ_g} \ ^gM^\epsilon &\To& kG
\hat{\otimes}_{kZ_g} \ ^gM^\epsilon, \\ f \otimes e \hat{\otimes} m
&\mapsto& \Tilde{\Phi}(f,e,g)fe \hat{\otimes} m . \nonumber
\end{eqnarray}
By direct verification one can show that

\begin{lemma}
The triple $(kG \hat{\otimes}_{kZ_g} \ ^gM^\epsilon, m_{_L},
d_{_L})$ makes a twisted crossed module for $(kG,\Phi)$ and we have
the following canonical isomorphism
\begin{equation}
^{C_g}M^{\epsilon} \cong kG \hat{\otimes}_{kZ_g} \ ^gM^{\epsilon} \
.
\end{equation}
\end{lemma}

We are now in a position to state the main result. First we
abbreviate some notations. For each $C \in \C,$ let $Z_C$ denote the
centralizer of one element in $C,$ say $g(C),$ and $\Tilde{\Phi}_C$
the corresponding 2-cocycle $\Tilde{\Phi}_{_{g(C)}}$ on $\Z_C$ as
defined in (3.8), and $M_C$ the $\Tilde{\Phi}_C$-representation
$^{g(C)}M^\epsilon$ of $Z_C.$ Let
$(kZ_C,\Tilde{\Phi}_C)\mathrm{-rep}$ denote the category of
$\Tilde{\Phi}_C$-representations of $Z_C$ and $\prod_{C \in \C}
(kZ_C,\Tilde{\Phi}_C)\mathrm{-rep}$ their cartesian product. By
$k^{\Tilde{\Phi}_C}Z_C$ we denote the twisted group algebra and by
$k^{\Tilde{\Phi}_C}Z_C\mathrm{-mod}$ we denote its left module
category. Let $\prod_{C \in \C} k^{\Tilde{\Phi}_C}Z_C\mathrm{-mod}$
denote the cartesian product of categories.

\begin{theorem}
We have the following equivalence of categories:
$$\mathcal{MB}(kG,\Phi) \cong \prod_{C \in \C}
(kZ_C,\Tilde{\Phi}_C)\mathrm{-rep} \cong \prod_{C \in \C}
k^{\Tilde{\Phi}_C}Z_C\mathrm{-mod} \ .$$
\end{theorem}

\begin{proof}
The latter equivalence is known (see \cite{karp}). We prove the
former one. Define a functor
$$\Theta: \mathcal{MB}(kG,\Phi) \To \prod_{C \in \C}
(kZ_C,\Tilde{\Phi}_C)\mathrm{-rep}$$ by $M \mapsto (M_C)_{C \in
\C}.$ For the converse, define a functor
$$\Xi: \prod_{C \in \C} (kZ_C,\Tilde{\Phi}_C)\mathrm{-rep} \To
\mathcal{MB}(kG,\Phi)$$ by $(V_C)_{C \in \C} \mapsto \bigoplus_{C
\in \C} kG \hat{\otimes}_{kZ_C} V_C \otimes kG.$ The
$(kG,\Phi)$-Majid bimodule structure on the latter is furnished by
(3.9)-(3.10) and (2.10)-(2.11). Straightforward but tedious
calculation shows that the functors $\Theta$ and $\Xi$ provide the
desired category equivalence.
\end{proof}

We should note that the collections of projective representations
which give rise to Majid bimodules are special, as the occurred
2-cocycles are uniformly determined by a 3-cocycle. For later
application, these collections are said to be admissible. When the
3-cocylcles need to be emphasized, those corresponding to
$(kG,\Phi)$-Majid bimodules are called $\Phi$-collections.

We remark that, by the previous category equivalence, it is easy to
derive some interesting consequences in the case of $G$ being a
finite group. By Maschke's theorem and its generalization, the
category $\mathcal{MB}(kG,\Phi)$ is semisimple if and only if the
character of the field $k$ does not divide the order of $G.$ When
$k$ is algebraically closed with characteristic 0, the set of simple
$(kG,\Phi)$-Majid bimodules is in one-to-one correspondence with the
set of simple $k^{\Tilde{\Phi}_C}Z_C$-modules, which in turn is in
one-to-one correspondence with the $\Tilde{\Phi}_C$-regular
conjugacy classes of $Z_C$ for all $C \in \C$ (see \cite{karp}). We
would also like to remark that the category $\mathcal{MB}(kG,\Phi)$
can be described by the module category of the twisted quantum
double $D^{\Phi}(G)$ introduced in \cite{dpr}. This is done by
generalizing the corresponding results for Hopf bimodules in
\cite{rosso1,sch1} to the Majid setting. The method adopted in this
paper is more elementary and straightforward, since the
representation theory of a group is much easier than that of the
complicated twisted quantum double.

\section{Quiver Majid Algebras and Gauge Equivalence}

As applications of our main result, in this section we give a
classification of the graded Majid algebras on the path coalgebras
of Hopf quivers and a description for the gauge equivalence.

Fix a group $G$ and a 3-cocycle $\Phi.$ Let $\S(g)=g^{-1},$
$\alpha(g)=1$ and $\beta(g)=1/\Phi(g,g^{-1},g)$ for all $g \in G.$
We understand $(kG,\Phi)$ as a Majid algebra with quasi-antipode
$(\S,\alpha,\beta).$ A ramification datum of the group $G$ is a
formal sum $R=\sum_{C \in \C} R_CC$ of conjugacy classes with
non-negative integer coefficients. Let $Q(G,R)$ denote the
corresponding Hopf quiver (see \cite{cr2}): the set of vertices is
$G,$ and for each $x \in G$ and $c \in C,$ there are $R_C$ arrows
going from $x$ to $cx.$ It is proved in \cite{qha1} that

\begin{proposition}
Let $G$ be a group, $\Phi$ a 3-cocycle, $R$ a ramification datum,
and $Q=Q(G,R)$ the associated Hopf quiver. Then the path coalgebra
$kQ$ admits a graded Majid algebra structure with $kQ_0 \cong
(kG,\Phi,\S,\alpha,\beta)$ as a sub Majid algebra if and only if
$kQ_1$ admits a $(kG,\Phi)$-Majid bimodule structure. Moreover, the
set of such graded Majid algebra structures on the path coalgebra
$kQ$ is in one-to-one correspondence with the set of
$(kG,\Phi)$-Majid bimodule structures on $kQ_1.$
\end{proposition}

As a corollary of Theorem 3.3 and Proposition 4.1, we have

\begin{theorem}
Let $G$ be a group, $\Phi$ a 3-cocycle, $R=\sum_{C \in \C} R_CC$ a
ramification datum, and $Q=Q(G,R)$ the associated Hopf quiver. Then
the set of graded Majid algebra structures on the path coalgebra
$kQ$ with $kQ_0 \cong (kG,\Phi,\S,\alpha,\beta)$ as a sub Majid
algebra is in one-to-one correspondence with the set of collections
$(M_C)_{C \in \C}$ where $M_C$ is a $\Tilde{\Phi}_C$-representation
of $Z_C$ of dimension $R_C.$
\end{theorem}

Note that we do not lose much of the generality by fixing the
quasi-antipode of $(kG,\Phi)$ as $(\S,\alpha,\beta)$ defined above,
since on the one hand the quasi-antipodes have nothing to do with
the Majid bimodules; on the other hand by \cite{d} an arbitrary
quasi-antipode is of the form $$(\mu \ast \S \ast \mu^{-1}, \quad
\mu \ast \alpha,\quad \beta \ast \mu^{-1})$$ where $\ast$ is the
convolution product and $\mu \in \operatorname{Hom}_k(kG,k)$ is a
convolution invertible function. In addition, cofree pointed
coalgebras can always be presented as path coalgebras (see
\cite{cr2}), so the preceding theorem amounts to a complete
classification of graded cofree pointed Majid algebras.

In the rest of this section, we study the gauge equivalence of
graded pointed Majid algebras. First we recall the definition. Let
$(H,\Phi,\S,\alpha,\beta)$ be a Majid algebra. A gauge
transformation on $H$ is a convolution-invertible 2-cochain $F: H
\otimes H \To k$ obeying $F(h,1_H)=F(1_H,h)=\e(h)$ for all $h \in
H.$ Given a gauge transformation $F$ on $H,$ there is a new Majid
algebra $H_F$ with new product, $\Phi, \alpha, \beta$ given by
\begin{gather}
g \cdot_{_F} h= F^{-1}(g_1,h_1)g_2 \cdot
h_2F(g_3,h_3), \\
\Phi_F(f,g,h)= \\ F^{-1}(g_1,h_1)F^{-1}(f_1,g_2 \cdot
h_2)\Phi(f_2,g_3,h_3)F(f_3 \cdot g_4,h_4)F(f_4,g_5),\nonumber \\
\alpha_F(h)=F(\S(h_1),h_3)\alpha(h_2), \quad
\beta_F(h)=F^{-1}(h_1,\S(h_3))\alpha(h_2)
\end{gather}
for all $f,g,h \in H.$ Two Majid algebras $H$ and $H'$ are said to
be gauge equivalent if there exists a gauge transformation $F$ on
$H$ such that $H_F$ and $H$ are isomorphic Majid algebras. The gauge
equivalence is an equivalence relation.

We introduce a gauge equivalence on Majid bimodules over a group $G$
and use it to describe the gauge equivalence of the associated Majid
algebras. If a 2-cochain $F: G \times G \To k^*$ is
convolution-invertible and satisfies $F(g,\epsilon)=F(\epsilon,g)=1$
for all $g \in G,$ we call it a gauge transformation on $G.$ Define
$\Phi_F: G \times G \times G \To k^*$ by
\begin{equation}
\Phi_F(e,f,g)=\frac{F(e,f)F(ef,g)}{F(e,fg)F(f,g)}\Phi(e,f,g).
\end{equation} Clearly $\Phi_F=\Phi  \partial F,$ where $\partial F$
is the differential of $F$ and the product is the convolution
product. So it is a 3-cocycle and cohomologous to $\Phi.$ Given a
$(kG,\Phi)$-Majid bimodule $M$ and a gauge transformation $F$ on
$G,$ we can build a $(kG,\Phi_F)$-Majid bimodule $M_F$ as follows.
The space and the $kG$-bicomodule are the same as $M,$ while the new
quasi-bimodule structure maps are given by
\begin{equation}
f ._{_F}m=\frac{F(f,h)}{F(f,g)}f.m, \quad m ._{_F} f
=\frac{F(h,f)}{F(g,f)}m.f
\end{equation} for all $f,g,h \in G$ and $m \in \ ^gM^h.$
It is easy to verify that $M_F$ is indeed a $(kG,\Phi_F)$-Majid
bimodule. We say that the $(kG,\Phi)$-Majid bimodule $M$ and the
$(kG,\Phi')$-Majid bimodule $N$ are gauge equivalent, if there
exists a gauge transformation $F$ on $G$ such that $\Phi'=\Phi_F$
and $M_F \cong N$ as $(kG,\Phi')$-Majid bimodules. If $F$ is a gauge
transformation on $G,$ then so is $F^{-1}$ and we have
\begin{equation} (M_F)_{F^{-1}} = M = (M_{F^{-1}})_F. \end{equation}
If $F'$ is another gauge transformation on $G,$ then so is the
product $FF'$ and we have \begin{equation} (M_F)_{F'}=M_{FF'}.
\end{equation} By (4.6)-(4.7), it is clear that the gauge equivalence is
an equivalence relation in the family of Majid bimodules over the
group $G$ with arbitrary 3-cocycles.

According to Theorem 3.3, the gauge equivalence of Majid bimodules
should be described by a suitable transformation of the family of
admissible collections of projective representations. Before giving
this transformation, we need to fix some notations. For each
conjugacy class $C$ of $G,$ fix an element $g(C) \in C.$ Then we
adopt all the notations as given before Theorem 3.3. Let $(M_C)_{C
\in \C}$ be a $\Phi$-collection. Suppose that $F$ is a gauge
transformation on $G.$ Let $\Phi_F$ be as (4.4). We build a
$\Phi_F$-collection $(M^F_C)_{C \in \C}$ by these data. The space of
$M^F_C$ is the same as $M_C$ and the action is given by
\begin{equation}
f \triangleright_{_F} m
=\frac{F(f,f^{-1})}{F(f,g(C))F(fg(C),f^{-1})} f \triangleright m
\end{equation}
for all $f \in Z_C$ and $m \in M_C.$ Let $\Phi'$ be another
3-cocycle on $G$ and $(N_C)_{C \in \C}$ a $\Phi'$-collection. The
two admissible collections $(M_C)_{C \in \C}$ and $(N_C)_{C \in \C}$
are said to be gauge equivalent, if there exists a gauge
transformation $F$ on $G$ such that $\Phi'=\Phi_F$ and $M^F_C \cong
N_C$ as $\Tilde{\Phi'}_C$-representations of $Z_C$ for all $C.$ It
is easy to verify that this gauge equivalence has properties like
(4.6)-(4.7) and therefore it is also an equivalence relation.

The equation (4.8) suggests that the gauge equivalence of admissible
collections can be described by the projective equivalence of
projective representations of groups. Now we need to recall some
more definitions and facts. Two projective representations $\rho_i:
G \To GL(V_i), \ \ i=1,2,$ are said to be projectively equivalent if
there exists a mapping $\mu: G \To k^*$ with $\mu(\epsilon)=1$ and a
vector space isomorphism $\phi: V_1 \To V_2$ such that
\begin{equation} \rho_2(g)=\mu(g)\phi\rho_1(g)\phi^{-1}
\end{equation} for all $g \in G.$ If $\mu(g)=1$ for all $g \in G,$
then $\rho_1$ and $\rho_2$ are said to be linearly equivalent. We
declare that the isomorphisms we have used so far are linear
equivalence by default. Let $\rho_i, \ \ i=1,2,$ be an
$\alpha_i$-representation. Then $\rho_1$ is projectively
(respectively, linearly) equivalent to $\rho_2$ implies that
$\alpha_1$ is cohomologous (respectively, equal) to $\alpha_2.$ On
the other hand, if $\rho_1$ is an $\alpha_1$ representation on the
space $V,$ then for any 2-cocycle $\alpha_2$ that is cohomologous to
$\alpha_1$ there exists an $\alpha_2$-representation $\rho_2$ on $V$
which is projectively equivalent to $\rho_1.$ In particular, if
$\alpha_1$ is a 2-coboundary, then $\rho_1$ is projectively
equivalent to a linear representation.

Let $F$ be a gauge transformation on $G,$ $(M_C)_{C \in \C}$ a
$\Phi$-collection, and $(M^F_C)_{C \in \C}$ the corresponding
$\Phi'$-collection under gauge transformation. For each $C \in \C,$
set a mapping $\mu^F_C: Z_C \To k^*$ by
\begin{equation}
\mu^F_C(f)=\frac{F(f,f^{-1})}{F(f,g(C))F(fg(C),f^{-1})}
\end{equation}
for all $f \in Z_C.$ Clearly, for all $C \in \C,$ $M_C$ and $M^F_C$
are projectively equivalent as projective representations of $Z_C.$
In particular, the mappings $\mu^F_C$ are determined by $F$ in a
uniform way. We say the two collections $(M_C)_{C \in \C}$ and
$(M^F_C)_{C \in \C}$ of projective representations are $F$-uniformly
projective equivalent. In general, when the gauge transformation is
not explicit, we simply say two collections are uniformly projective
equivalent. In this terminology, a $\Phi$-collection $(M_C)_{C \in
\C}$ is gauge equivalent to a $\Phi'$-collection $(N_C)_{C \in \C}$
if and only if there exists a gauge transformation $F$ on $G$ such
that $\Phi'=\Phi
\partial F$ and that the two collections $(M_C)_{C \in \C}$ and
$(N_C)_{C \in \C}$ are $F$-uniformly projective equivalent.

Now we are ready to describe the gauge equivalence of graded Majid
algebras. For a Majid bimodule $M$ over $G,$ let $(M_C)_{C \in \C}$
be the corresponding admissible collection. We have shown in
\cite{qha1} that one can associate to $M$ a unique Hopf quiver and
construct a unique graded Majid algebra, denoted by $H(M),$ on the
path coalgebra as in Proposition 4.1. Similar to Nichols' bialgebras
of type one \cite{nichols},  we define $B(M)$ to be the sub Majid
algebra of $H(M)$ generated in degrees 0 and 1, that is the smallest
sub Majid algebra of $H(M)$ containing $G$ and $M.$

\begin{theorem}
Let $M,N$ be Majid bimodules over $G.$ Then the following are
equivalent:
\begin{enumerate}
  \item $H(M)$ and $H(N)$ are gauge equivalent.
  \item $B(M)$ and $B(N)$ are gauge equivalent.
  \item $M$ and $N$ are gauge equivalent.
  \item $(M_C)_{C \in \C}$ and $(N_C)_{C \in \C}$ are gauge equivalent.
  \item $(M_C)_{C \in \C}$ and $(N_C)_{C \in \C}$ are uniformly projective
        equivalent.
\end{enumerate}
\end{theorem}

\begin{proof} $(3) \Leftrightarrow (4)$ and $(4) \Leftrightarrow (5)$ are
direct from definitions. $(1) \Leftrightarrow (3)$ and $(2)
\Leftrightarrow (3)$ are similar, so we prove only $(1)
\Leftrightarrow (3).$

First suppose that $M$ is a $(kG,\Phi)$-Majid bimodule and $N$ is a
$(kG,\Phi')$-Majid bimodule and that they are gauge equivalent. Then
there exists a gauge transformation $F$ on $G$ such that
$\Phi'=\Phi_F$ and $M_F \cong N$ as $(kG,\Phi')$-Majid bimodules.
Clearly $F$ can be lifted as a gauge transformation on $H(M)$ which
concentrates at degree zero. By the defintion of $H(M)_F,$ one sees
that it is exactly the graded Majid algebra corresponding to the
Majid bimodule $M_F.$ Now we have $H(M)_F = H(M_F) \cong H(N)$ by
Proposition 4.1.

Conversely, suppose that $H(M)$ and $H(N)$ are gauge equivalent.
Then there exists a gauge transformation $F$ on $H(M)$ such that
$H(M)_F$ is isomorphic to $H(N).$ Since $H(M)_F$ is graded, then by
(4.1) we have that $F$ concentrates at degree zero part. When we
restrict $F$ to degree zero part, the equation (4.2) becomes (4.4),
so it is actually a gauge transformation on $G.$ The isomorphism map
preserves the coradical filtration. In particular its restriction to
the coradicals, i.e. the degree zero parts, is an isomorphism.
Without loss of generality, we may assume the map preserves
gradation and the restriction on degree zero is identity. Assume
that $M$ is a $(kG,\Phi)$-Majid bimodule and $N$ and a
$(kG,\Phi')$-Majid bimodule. Now by (4.2) we have $\Phi'=\Phi_F.$
The isomorphism $H(M)_F \cong H(N)$ also induces the isomorphism of
degree one parts as $(kG,\Phi')$-Majid bimodules. Since $H(M)_F =
H(M_F),$ we have $M_F \cong N.$
\end{proof}

The previous result can be generalized to the gauge equivalence for
more general graded pointed Majid algebras. For example, The
equivalent conditions (3)-(5) are always necessary for two graded
pointed Majid algebras to be gauge equivalent. Now it is interesting
to remark that, the isomorphism is to the gauge equivalence for
graded pointed Majid algebras is almost as much as the linear
equivalence is to the projective equivalence for projective
representations of groups.

We conclude this section with a corollary of the previous theorem,
which deals with the special case when a graded pointed Majid
algebra is gauge equivalent to a usual Hopf algebra.

\begin{corollary}
Suppose that $M$ is a $(kG,\Phi)$-Majid bimodule. Let $(M_C)_{C \in
\C},$ $H(M)$ and $B(M)$ be as before. Then the following are
equivalent:
\begin{enumerate}
  \item $H(M)$ is gauge equivalent to a Hopf algebra.
  \item $B(M)$ is gauge equivalent to a Hopf algebra.
  \item $M$ is gauge equivalent to a $kG$-Hopf bimodule.
  \item $\Phi$ is a 3-coboundary.
  \item $(M_C)_{C \in \C}$ is gauge equivalent to an admissible collection of linear representations.
  \item $(M_C)_{C \in \C}$ is uniformly projective equivalent to a collection of linear representations.
\end{enumerate}
\end{corollary}

\begin{proof}
We prove only $(1) \Leftrightarrow (4),$ since the other
equivalences are direct consequence of Theorem 4.3. Suppose that
$\Phi$ is a 3-coboundary, then there exists a 2-cochain $F$ such
that $\Phi=\partial F.$ Note that $F$ can be chosen in the gauge
transformations on $G.$ By (4.4) we have that $\Phi_{F^{-1}}$ is the
trivial 3-cocycle and so the Majid bimodule $M_{F^{-1}}$ defined by
by (4.5) is a $kG$-Hopf bimodule. Now it is clear that
$H(M)_{F^{-1}}=H(M_{F^{-1}})$ is a Hopf algebra. That is, $H(M)$ is
gauge equivalent to a Hopf algebra. For the converse, suppose that
$H(M)$ is gauge equivalent to a Hopf algebra $H.$ There exists a
gauge transformation $F$ on $H(M)$ such that $H(M)_F \cong H.$ Note
that the reassociator on $H$ is trivial, and then by (4.2) we have
that the 3-cocycle $\Phi_F$ on $G$ is trivial, which means that
$\Phi=\partial F^{-1}.$ We are done.
\end{proof}

\section{Majid algebras over $\Z_2$}

In this section, we apply the constructed machinery to give some
examples and a classification of finite-dimensional graded pointed
Majid algebras over the simplest nontrivial group $\Z_2.$

For the construction of graded Majid algebras on path coalgebras, we
have to calculate the product of paths. This can be done by the
quantum shuffle product as shown in \cite{qha1}. So firstly we need
to recall the the product formula in \cite{cr2}.

Suppose that $Q$ is a Hopf quiver with a necessary $kQ_0$-Majid
bimodule structure on $kQ_1.$ By $Q_l$ we denote the set of paths of
length $l.$ Let $p \in Q_l$ be a path. An $n$-thin split of it is a
sequence $(p_1, \ \cdots, \ p_n)$ of vertices and arrows such that
the concatenation $p_n \cdots p_1$ is exactly $p.$ These $n$-thin
splits are in one-to-one correspondence with the $n$-sequences of
$(n-l)$ 0's and $l$ 1's. Denote the set of such sequences by
$D_l^n.$ Clearly $|D_l^n|={n \choose l}.$ For $d=(d_1, \ \cdots, \
d_n) \in D_l^n,$ the corresponding $n$-thin split is written as
$dp=((dp)_1, \ \cdots, \ (dp)_n),$ in which $(dp)_i$ is a vertex if
$d_i=0$ and an arrow if $d_i=1.$ Let $\alpha=a_m \cdots a_1$ and
$\beta=b_n \cdots b_1$ be paths of length $m$ and $n$ respectively.
Let $d \in D_m^{m+n}$ and $\bar{d} \in D_n^{m+n}$ the complement
sequence which is obtained from $d$ by replacing each 0 by 1 and
each 1 by 0. Define an element
$$(\alpha \cdot \beta)_d=[(d\alpha)_{m+n}.(\bar{d}\beta)_{m+n}] \cdots
[(d\alpha)_1.(\bar{d}\beta)_1]$$ in $kQ_{m+n},$ where
$[(d\alpha)_i.(\bar{d}\beta)_i]$ is understood as the action of
$kQ_0$-Majid bimodule on $kQ_1$ and these terms in different
brackets are put together by cotensor product, or equivalently
concatenation. In terms of these notations, the formula of the
product of $\alpha$ and $\beta$ is given as follows:
\begin{equation}
\alpha \cdot \beta=\sum_{d \in D_m^{m+n}}(\alpha \cdot \beta)_d \ .
\end{equation}

We should remark that, for general Majid algebras, the product is
not associative. So the order must be concerned for the product of
more than two terms.

\begin{convention}
For an arbitrary path $p$ and an integer $n \ge 3,$ let
$p^{\stackrel{\rightharpoonup}{n}}$ denote the product
$\stackrel{n-2}{\overbrace{( \cdots (}}p \cdot p) \cdot \cdots \cdot
p.$ For consistency, when $n < 3,$ we still use the notation
$p^{\stackrel{\rightharpoonup}{n}}$ although there is no risk of
associative problem.
\end{convention}

For simplicity, we assume that $k$ {\em is the field of complex
numbers} in the rest of the paper. Let $g$ denote the generator of
the group $\Z_2.$ There is only one nontrivial 3-cocycle $\Phi$ on
$\Z_2$ by (2.5)-(2.6). That is, the value of $\Phi$ on the triple
$(g,g,g)$ is $-1$ and on any other triples is 1. The induced
2-cocycle $\Tilde{\Phi}_\epsilon$ is trivial, while $\Tilde{\Phi}_g$
is nontrivial satisfying $\Tilde{\Phi}_g(g,g)=-1.$ For the
classification of $(k\Z_2,\Phi)$-Majid bimodules, it suffices to
study the representation theory of $k^{\Tilde{\Phi}_\epsilon}\Z_2 =
k\Z_2$ and $k^{\Tilde{\Phi}_g}\Z_2.$ The former is well-known. For
the latter, let $\star$ denote the product and then we have
\begin{equation} g \star g = \Tilde{\Phi}_g(g,g)g^2 = -\epsilon . \end{equation}
Therefore the action of $g$ on a linear space is always
diagonalizable and so simple $k^{\Tilde{\Phi}_g}\Z_2$-modules are
1-dimensional. Assume $V=kv$ is a simple
$k^{\Tilde{\Phi}_g}\Z_2$-module and $g \triangleright v = \lambda
v.$ It follows from (5.2) that $\lambda^2=-1,$ hence $\lambda=\pm
i.$ There are two simple $k^{\Tilde{\Phi}_g}\Z_2$-modules, denoted
by $S(+i)$ and $S(-i)$ respectively. An arbitrary
$k^{\Tilde{\Phi}_g}\Z_2$-module is of the form $mS(+i) \oplus
nS(-i)$ where $m,n$ are non-negative integers.

Let $R_1=g$ be a ramification datum of $\Z_2.$ Then the associated
Hopf quiver $Q(\Z_2,R_1),$ denoted by $\Gamma^1$ for short, is as
follows:
$$ \xy (0,0)*{\epsilon}; (50,0)*{g.}; (25,3)*{X}; (25,-3.3)*{Y}; {\ar
(2,1)*{}; (47,1)*{}}; {\ar (47,-1)*{};(2,-1)*{}};
\endxy $$ Conversely, this is the unique way for $\Gamma^1$ to be viewed as a Hopf
quiver. So the set of graded Majid algebra structures on the path
coalgebra $k\Gamma^1$ is determined by all the possible Majid
bimodule structures on $M=kX \oplus kY$ over the group $\Z_2.$ We
are interested in nontrivial Majid algebras (i.e., not gauge
equivalent to Hopf algebras), so we take the nontrivial 3-cocycle
$\Phi$ on $\Z_2$ as above. Note that $\Phi$ is not a 3-coboundary by
the well-known fact $H^3(\Z_2,k^*)=\Z_2.$ By Theorem 3.3, it is
enough to consider the $k^{\Tilde{\Phi}_g}Z_2$-module structure on
$^gM^\epsilon=kX,$ which is either $S(+i)$ or $S(-i)$ by the
previous argument. It is clear that these two modules are not gauge
equivalent, since by (4.8) they are stable under any gauge
transformation of $\Z_2.$ Hence there are two non-trivial graded
Majid algebras on $k\Gamma^1$ up to gauge equivalence described as
follows.

\begin{example}
The one induced by $S(+i),$ denoted by $k\Gamma^1(+i).$ By
Theorem 3.3, we may extend $S(+i)$ to a Majid bimodule on $M$ as:
  \begin{equation}
  X.g=Y, \quad Y.g=-X, \quad g.X=iY, \quad g.Y=iX.
  \end{equation}
We have the following interesting product formulae of paths:
\begin{equation}
X \cdot X = (1+i)YX, \quad YX \cdot X = -iXYX, \quad YX \cdot YX =
XYX \cdot X =0.
\end{equation}
Consider the sub Majid algebra generated by the vertices and arrows
of $\Gamma^1.$ By (5.3) it is actually generated by $g$ and $X.$
Since $((X \cdot X) \cdot X) \cdot X)=0$ and the products in other
orders differ only by nonzero scalars according to the definition of
Majid algebras, then along with (5.4) it is easy to show that all
the paths of length less than 4 constitute a basis for this sub
Majid algebra. In particular it is 8-dimensional. Let $M_+(8)$
denote this Majid algebra.
\end{example}

\begin{example}
The one induced by $S(-i),$ denoted by $k\Gamma^1(-i).$ We may
extend $S(-i)$ to a Majid bimodule on $M$ as:
  \begin{equation}
  X.g=Y, \quad Y.g=-X, \quad g.X=-iY, \quad g.Y=-iX.
  \end{equation}
We have the following similar product formulae of paths:
\begin{equation}
X \cdot X = (1-i)YX, \quad YX \cdot X = iXYX, \quad YX \cdot YX =
XYX \cdot X =0.
\end{equation}
Similarly the sub Majid algebra generated by the vertices and arrows
of $\Gamma^1$ is 8-dimensional and all the paths of length less than
4 constitute a basis. Let $M_-(8)$ denote this Majid algebra.
\end{example}

Now let $R_2=2g$ be a ramification datum of $\Z_2$ and $\Gamma^2$
the associated Hopf quiver: $$ \xy (0,0)*{\epsilon}; (50,0)*{g.};
{\ar (2,1)*{}; (47,1)*{}}; {\ar (47,-1)*{};(2,-1)*{}}; {\ar
(2,2)*{}; (47,2)*{}}; {\ar (47,-2)*{};(2,-2)*{}};
\endxy $$ Let $X_1,X_2$ denote the two arrows going from $\epsilon$
to $g,$ and $Y_1,Y_2$ those in the converse direction. The set of
nontrivial graded Majid algebra structures on $k\Gamma^2$ is
determined by all the $k^{\Tilde{\Phi}_g}Z_2$-module structures on
$^gM^\epsilon = kX_1 \oplus kX_2,$ which are $S(+i) \oplus S(+i), \
S(-i) \oplus S(-i),$ and $S(+i) \oplus S(-i).$ Hence there are three
non-trivial graded Majid algebras on $k\Gamma^2,$ denoted by
$k\Gamma^2(++), \ k\Gamma^2(--),$ and $k\Gamma^2(\pm)$ respectively.

\begin{example}
The product of the Majid algebra $k\Gamma^2(\pm)$ is determined by
the Majid bimodule extended by $S(+i) \oplus S(-i).$ Without loss of
generality, we may assume
\begin{gather}
X_1.g=Y_1, \quad Y_1.g=-X_1, \quad g.X_1=iY_1, \quad g.Y_1=iX_1, \\
X_2.g=Y_2, \quad Y_2.g=-X_2, \quad g.X_2=-iY_2, \quad g.Y_2=-iX_2.
\end{gather}
Then we have product formulae like (5.4) by replacing $X,Y$ as
$X_1,Y_1$ and (5.6) by replacing $X,Y$ as $X_2,Y_2.$ We also have
\begin{equation}
X_1 \cdot X_2 = Y_1X_2 - iY_2X_1, \quad X_2 \cdot X_1 = iY_1X_2 +
Y_2X_1.
\end{equation} This implies that \begin{equation} X_1 \cdot X_2 =-iX_2
\cdot X_1. \end{equation} Consider the sub Majid algebra generated
by the vertices and arrows of $\Gamma^2$ and denote it by $M(32).$
By (5.7)-(5.8) it is actually generated by $g, X_1, X_2.$ According
to the previous product formulae, it is not difficult to verify that
$$\{ g^l \cdot (X_1^{\stackrel{\rightharpoonup}{m}} \cdot
X_2^{\stackrel{\rightharpoonup}{n}}) \ | \ 0 \le m,n \le 3, \ l=0,1
\}$$ constitutes a basis for $M(32).$ In particular $M(32)$ is
32-dimensional.
\end{example}

We derive some product formulae for $k\Gamma^2(++)$ and
$k\Gamma^2(--)$ in the same manner.

\begin{example}
Extend $S(+i) \oplus S(+i)$ to the Majid bimodule as follows:
\begin{gather}
X_1.g=Y_1, \quad Y_1.g=-X_1, \quad g.X_1=iY_1, \quad g.Y_1=iX_1,  \\
X_2.g=Y_2, \quad Y_2.g=-X_2, \quad g.X_2=iY_2, \quad g.Y_2=iX_2.
\end{gather}
Then we have product formulae for $k\Gamma^2(++)$ like (5.4) by
replacing $X,Y$ as $X_j,Y_j,j=1,2.$ We also have
\begin{equation}
X_1 \cdot X_2 = Y_1X_2 + iY_2X_1, \quad X_2 \cdot X_1 = iY_1X_2 +
Y_2X_1.
\end{equation} This implies that \begin{equation} X_2 \cdot X_1 -
iX_1 \cdot X_2 = 2Y_2X_1. \end{equation}
\end{example}

\begin{example}
Extend $S(-i) \oplus S(-i)$ to the Majid bimodule as follows:
\begin{gather}
X_1.g=Y_1, \quad Y_1.g=-X_1, \quad g.X_1=-iY_1, \quad g.Y_1=-iX_1,  \\
X_2.g=Y_2, \quad Y_2.g=-X_2, \quad g.X_2=-iY_2, \quad g.Y_2=-iX_2.
\end{gather}
Then we have product formulae for $k\Gamma^2(--)$ like (5.6) by
replacing $X,Y$ as $X_j,Y_j,j=1,2.$ We also have
\begin{equation}
X_1 \cdot X_2 = Y_1X_2 - iY_2X_1, \quad X_2 \cdot X_1 = Y_2X_1 -
iY_1X_2.
\end{equation} This implies that \begin{equation} X_1 \cdot X_2 + iX_2
\cdot X_1 = 2Y_1X_2. \end{equation}
\end{example}

We are interested in their sub Majid algebras generated by the
vertices and arrows of $\Gamma^2.$ It turns out that the situation
is quite different from that of $k\Gamma^2(\pm).$

\begin{proposition}
The sub Majid algebras of $k\Gamma^2(++)$ and $k\Gamma^2(--)$
generated by the set of vertices and arrows of $\Gamma^2$ are
infinite-dimensional.
\end{proposition}

\begin{proof}
We only prove the claim for $k\Gamma^2(++).$ The other is similar.
By (5.14), the path $Y_2X_1$ lies in the sub Majid algebra.
Calculate the power $(Y_2X_1)^{\stackrel{\rightharpoonup}{n}}.$ By
(5.1) and induction on $n,$  it is not hard to show that
\begin{equation}
(Y_2X_1)^{\stackrel{\rightharpoonup}{n}}=n!\stackrel{\textrm{n
copies of $Y_{2}X_{1}$}} {\overbrace{Y_{2}X_{1}\; \cdots\;
Y_{2}X_{1}}}+\; \textrm{other paths}.
\end{equation}
This implies that $(Y_2X_1)^{\stackrel{\rightharpoonup}{n}} \ne 0$
for all $n$ and that the set
$\{(Y_2X_1)^{\stackrel{\rightharpoonup}{n}} \ | \ n \ge 0 \}$ is
linearly independent. We are done.
\end{proof}

An arbitrary ramification datum of $\Z_2$ is of the form
$R=m\epsilon+ng.$ The associated Hopf quiver looks like the
following
$$ 
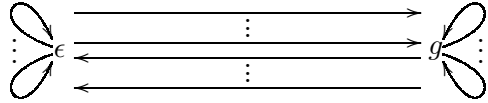
\xy (0,0)*{\epsilon}; (50,0)*{g}; (25,2)*{.}; (25,3)*{.}; (25,4)*{.};
(25,-2)*{.}; (25,-3)*{.}; (25,-4)*{.}; {\ar (2,1)*{}; (48,1)*{}};
{\ar (48,-1)*{};(2,-1)*{}}; {\ar (2,5)*{}; (48,5)*{}}; {\ar
(48,-5)*{};(2,-5)*{}}; (51,0.5)*{}="A";(51,1.5)*{}="B"; "A";
"B"**\crv{(54,1.5) & (59,5) & (54,8)} ?(1)*\dir{>};
(51,-0.5)*{}="A";(51,-1.5)*{}="B"; "A"; "B"**\crv{(54,-1.5) &
(59,-5) & (54,-8)} ?(1)*\dir{>}; (56,1)*{\vdots};
(-1,0.5)*{}="A";(-1,1.5)*{}="B"; "A"; "B"**\crv{(-4,1.5) & (-9,5) &
(-4,8)} ?(1)*\dir{>}; (-1,-0.5)*{}="A";(-1,-1.5)*{}="B"; "A";
"B"**\crv{(-4,-1.5) & (-9,-5) & (-4,-8)} ?(1)*\dir{>};
(-6,1)*{\vdots};
\endxy $$
where there are $m$ loops attached to both vertices, and there are
$n$ arrows going from $\epsilon$ to $g$ and vice versa. As before,
one can apply the representation theory of twisted group algebras to
get Majid bimodule structures on the space of arrows and then the
product formulae for the graded Majid algebras.

Now we are ready to give our first classification result for
finite-dimensional graded pointed Majid algebras with a help of the
quiver framework.

\begin{theorem}
Assume that $H$ is a finite-dimensional graded pointed nontrivial
Majid algebra whose set of group-likes is equal to $\Z_2.$ Then $H$
is isomorphic to one of $\{M_+(8), \ M_-(8), \ M(32)\}.$ Moreover,
$M_+(8)$ is not gauge equivalent to $M_-(8).$
\end{theorem}

\begin{proof}
By assumption we can write $H=\oplus_{n \ge 0}H_n$ where
$H_0=k\Z_2.$ Let $\Phi$ denote the reassociator of $H.$ Since $\Phi$
concentrates at $M_0$ by the graded condition, it is a 3-cocycle on
$\Z_2.$ Moreover $\Phi$ is nontrivial by Corollary 4.4, since the
Majid algebra $H$ is not gauge equivalent to a Hopf algebra.

By the Gabriel type theorem for pointed Majid algebras, there exists
a unique Hopf quiver of the form $Q(\Z_2,R),$ where $R$ is a
ramification datum of $\Z_2,$ such that $H$ is isomorphic to a large
sub Majid algebra of a graded Majid algebra structure on
$kQ(\Z_2,R).$ Recall that ``large" means containing the subspace
$kQ(\Z_2,R)_0 \oplus kQ(\Z_2,R)_1.$ In the following we always view
$H$ as a sub Majid algebra of $kQ(\Z_2,R).$

We claim that the Hopf quiver $Q(\Z_2,R)$ of $H$ can only be
$\Gamma^1$ or $\Gamma^2.$ Assume that $R=m\epsilon+ng.$ If $m
> 0,$ then there is at least one loop attached to $\epsilon.$
Let $\ell$ denote the loop and we have $\ell \in H$ since $H$ is
large. Now we calculate $\ell^{\stackrel{\rightharpoonup}{s}}$ by
making use of (5.1). By induction on $s,$ we get
$\ell^{\stackrel{\rightharpoonup}{s}}=s!\stackrel{\textrm{n}}
{\overbrace{\ell \; \cdots \; \ell}}.$ It follows that the infinite
set $\{ \ell^{\stackrel{\rightharpoonup}{s}} \ | \ s \ge 0 \}$ is
linearly independent, which contradicts the finiteness assumption of
$H.$ If $n > 2,$ then there are at least three arrows going from
$\epsilon$ to $g.$ That means $kQ(\Z_2,R)$ must contain at least a
sub Majid algebra as $k\Gamma^2(++)$ or $k\Gamma^2(--).$ By
Proposition 5.7, the minimal large sub Majid algebra (generated by
the set of vertices and arrows) is already infinite-dimensional,
which contradicts the finiteness assumption again. Therefore, $m=0$
and $n \le 2,$ which justifies the claim.

Now by Examples (5.2)-(5.4) and Proposition (5.7) we can conclude
that $H$ must be a large sub Majid algebra of $k\Gamma^1(+i),
k\Gamma^1(-i)$ or $k\Gamma^2(\pm).$ Therefore $H$ must contain one
of $\{M_+(8), \ M_-(8), \ M(32)\}$ as a sub Majid algebra, since
these are minimal large sub Majid algebras. It remains to prove that
$M_+(8), \ M_-(8),$ and $M(32)$ are the only finite-dimensional
large sub Majid algebras of $k\Gamma^1(+i), \ k\Gamma^1(-i),$ and
$k\Gamma^2(\pm)$ respectively. We will verify this case by case.

Assume that $H \supsetneq M_+(8)$ is a graded sub Majid algebra of
$k\Gamma^1(+i).$ Then $H$ must contain paths of length $\ge 4.$
Consider the coproduct of such paths, then one can see that there
exist paths of length exactly equal to 4. So we have $YXYX \in H.$
Calculate $(YXYX)^{\stackrel{\rightharpoonup}{s}}.$ By induction on
$s$ we have $(YXYX)^{\stackrel{\rightharpoonup}{s}}=\stackrel{s \
\textrm{terms of} \ YXYX}{\overbrace{YXYX \ \cdots \ YXYX}}.$ It
follows that $H$ must be infinite-dimensional. Similarly we can
prove the claim for $M_-(8).$

Assume that $H \supsetneq M(32)$ is a finite-dimensional graded sub
Majid algegra of $k\Gamma^2(\pm).$ That means $g, X_1, X_2$ can not
fully generate $H.$ As the above argument, paths with segments like
$Y_1X_1Y_1X_1$ or $Y_2X_2Y_2X_2$ can not occur in $H.$ Consider the
linear combinations of $Y_2X_1$ and $Y_1X_2$ occurs in $H.$ If all
of them are linearly dependent to $Y_1X_2 - iY_2X_1 = X_1 \cdot
X_2,$ then $g, X_1, X_2$ generate $H.$ There must be linear
combinations of $Y_2X_1$ and $Y_1X_2$ which are linearly independent
to $Y_1X_2 - iY_2X_1.$ In that case, we have $Y_2X_1 \in H.$ Now
calculate $(Y_2X_1)^{\stackrel{\rightharpoonup}{s}}$ and get a
formula like (5.19), then we arrive at a contradiction.

The second claim is a direct consequence of Theorem 4.3 and the fact
that $S(+i)$ is not gauge equivalent to $S(-i).$ The proof is
completed.
\end{proof}

We conclude the paper with some remarks. Firstly, if we take $\Phi$
to be the trivial 3-cocycle on $\Z_2$ and carry out the previous
procedure, then we get graded pointed Hopf algebras. To get
finite-dimensional graded pointed Hopf algebras over $Z_2,$ we need
to work on Hopf quivers of the form $Q(\Z_2,ng).$ For all integers
$n,$ one can recover the well-known Nichols' Hopf algebras
\cite{nichols} of dimension $2^{n+1}$ on the Hopf quivers
$Q(\Z_2,ng),$ which provide a classification of finite-dimensional
graded pointed Hopf algebras over $\Z_2.$ For those who are familiar
with this result, Theorem 5.8 seems weird at first sight. We believe
that our approach gives an acceptable representation-theoretic
explanation for this phenomenon.

Secondly, note that the notions of Majid algebras and quasi-Hopf
algebras are dual to each other. In particular, the dual of a
finite-dimensional pointed Majid algebra is an elementary quasi-Hopf
algebras and vice versa. Our method can be dualized to get the
corresponding results on quasi-Hopf algebras by making use of path
algebras and quasi-Hopf bimodules (see \cite{sch2}) over the dual of
group algebras. For example, the dual of Theorem 5.8 recovers
Etingof and Gelaki's classification of elementary graded non-trivial
quasi-Hopf algebras with radical of codimension 2 in \cite{eg1}. It
is easy to check that the dual of our Majid algebras $M_+(8), \
M_-(8), \ M(32)$ are isomorphic to their quasi-Hopf algebras
$H_+(8), \ H_-(8), \ H(32)$ respectively. Moreover, the coproducts
of $H_+(8), \ H_-(8), \ H(32)$ can be given by the quasi-bicomodules
of the quasi-Hopf bimodules (i.e., the dual of our Majid bimodules
in Examples 5.2-5.4) in a conceptual way. It should be pointed out
that these two approaches are not just simply dual to each other.
Our approach emphasizes the coalgebraic side and the Majid algebras
are constructed as sub algebras, while the approach of Etingof and
Gelaki emphasizes the algebraic side and their quasi-Hopf algebras
are constructed as quotient algebras. It is better to combine both,
especially when we want to determine a finite-dimensional
quasi-quantum group. For example, in the proof of Theorem 5.8, our
approach gives easily the lower bound of the dimension of a pointed
Majid algebra in concern, while Etingof and Gelaki's approach gives
easily the upper bound of the dimension of the quasi-Hopf algebra
(=the dual of our Majid algebra), then the claim follows from the
dimension comparison.

Finally, we would like to remark that the quiver presentation for
Majid algebras and quasi-Hopf algebras is very useful in studying
their structure and (co)representation theory. For example, the
quiver presentations of $M_+(8)$ and $M_-(8)$ indicate that they are
monomial in the sense of \cite{chyz}. Then similar to [loc. cit.],
one can show by quiver representation theory that their comodule
categories have a very interesting property, namely there are only
finitely many indecomposable objects up to isomorphism. This
phenomenon is called finite representation type and plays very
import role in the representation theory of algebras (see
\cite{ass}). The study of such finite-dimensional algebras has been
a central theme all along. Thus in our setting, a very natural
problem is to give a complete classification of finite-dimensional
pointed Majid algebras of finite corepresentation type in the sense
of \cite{ll}. According to the Tannakian formalism, it amounts to a
classification of the class of finite tensor categories in which
there are finitely many indecomposable objects up to isomorphism
with mild conditions. We believe that this should be of interest for
the theory of finite tensor categories as well. This problem is
dealt with in another paper \cite{qha3}.

\vskip 0.5cm

\noindent{\bf Acknowledgement:} The research was supported by the
National NSF of China under grant number 10601052. The author thanks
Gongxiang Liu and Yu Ye for useful discussions.

\end{document}